\documentclass[12pt,oneside]{amsart}
\usepackage{amscd,amsmath,amssymb,amsfonts,a4}
\usepackage{hyperref}
\usepackage[cmtip, all]{xy}
\usepackage{url}
\usepackage{color}
 
\newcommand{\makered}[1]{}

\newlength{\rulebreite}

\def\timesover#1#2#3{\ \xymatrix@1@=0pt@M=0pt{ _{#1}&\times&_{#2} \\& ^{#3}&}\ }
\def\otimesover#1#2#3{\ \xymatrix@1@=0pt@M=0pt{ _{#1}&\otimes&_{#2} \\& ^{#3}&}\
}

\usepackage[all]{xy}
\theoremstyle{plain}
\newtheorem{thm}{Theorem}

\newtheorem{lemma}[thm]{Lemma}
\newtheorem{cor}[thm]{Corollary}

\theoremstyle{definition}
\newtheorem{defn}[thm]{Definition}

\newtheorem{rmk}[thm]{Remark}

\numberwithin{thm}{section}
\numberwithin{equation}{section}

\newcommand{\ga}[2]{\begin{gather}\label{#1}#2 \end{gather}}

\newcommand{\sE}{{\mathcal E}}

\newcommand{\sN}{{\mathcal N}}
\newcommand{\sO}{{\mathcal O}}

\newcommand{\sS}{{\mathcal S}}

\newcommand{\F}{{\mathbb F}}

\newcommand{\M}{{\mathbb M}}

\newcommand{\et}{{\acute{e}t}}

\def\tilde{\widetilde}
\begin{document}

\title[Stratifications in characteristic $p>0$]
{A relative version of Gieseker's problem on stratifications in 
characteristic $p>0$}

\author{H\'el\`ene Esnault}
\address{
Freie Universit\"at Berlin, Arnimallee 3, 14195, Berlin,  Germany}
\email{esnault@math.fu-berlin.de}
\author{Vasudevan Srinivas}
\address{School of Mathematics, Tata Institute of Fundamental Research, Homi
Bhabha Road, Colaba, Mumbai-400005, India} 
\email{srinivas@math.tifr.res.in}
\date{\today}
\thanks{The first  author is supported by  the Einstein program; the second 
author is supported by an Einstein Visiting 
Fellowship, and by a J.C. Bose Fellowship.}
\begin{abstract} 
We prove that the vanishing of the functoriality morphism for the \'etale 
fundamental group between smooth projective varieties over an algebraically 
closed field of characteristic $p>0$ forces the same property for the 
fundamental groups of stratifications. 

\end{abstract}
\maketitle
\section{Introduction}

In analogy with complex geometry, where according to Mal\v{c}ev-Grothendieck's 
theorem (\cite{Mal40}, \cite{Gro70}), a simply connected algebraic complex 
manifold has no non-trivial regular singular connections, Gieseker  conjectured 
in \cite[p.~8]{Gie75} that simply connected smooth projective varieties defined 
over an algebraically closed characteristic $p>0$ field have no non-trivial 
stratified bundles. 
This has been solved in the positive in \cite[Thm.~1.1]{EM10}.  Gieseker in {\it 
loc.cit.} also conjectured 
that if the commutator of the \'etale fundamental group is a pro-$p$-group, 
irreducible stratified bundles have rank $1$, and if  the \'etale fundamental 
group is abelian with no non-trivial $p$-power quotient, then the category of 
stratified bundles is semi-simple with irreducible objects of rank $1$. This is 
proven in \cite[Thm.~3.9]{ES14}.
In \cite[Thm.~3.2]{ES16} we extended \cite[Thm.~1.1]{EM10} to the regular locus 
of a normal variety, assuming in addition that the ground field is $\bar \F_p$. 

In this article we prove a relative version. 
\begin{thm}\label{Mainthm}
Let $\pi:Y\to X$ be a morphism between smooth projective varieties over an 
algebraically closed 
field $k$ of characteristic $p>0$, such that for any finite \'etale cover $Z\to 
X$, the pull-back cover $Y\times_XZ\to Y$ splits completely. 
Then for any stratified vector bundle
$\M$ 
on $X$, the pull-back stratified bundle $\pi^*\M$ on $Y$ is trivial.
\end{thm}
We can of course translate the theorem using the fundamental groups: if  the 
homomorphism
$f_*:\pi_1^{{\rm \acute{e}t}}(Y)\to \pi_1^{{\rm \et}}(X)$ on the \'etale 
fundamental groups  is the zero map, then so is the homomorphism 
$f_*: \pi_1^{\rm strat}(Y)\to \pi_1^{\rm strat}(X)$ 
on the $k$-Tannaka group-schemes  of stratified bundles (we omit all the base 
points as this is irrelevant here). 

As in the other proofs in {\it loc.cit.}, the key tool is the theorem of 
Hrushovski  \cite[Cor.~1.2]{Hru04} on the density of preperiodic points by a 
correspondence over $\bar \F_p$. A complete proof of  it in the framework of 
arithmetic geometry is due to Varshavsky \cite[Thm.~0.1]{Var14}. However, with 
this tool in the back ground, one needs a slightly different strategy as 
compared to the other proofs in {\it loc. cit.}. Indeed,  one can not argue 
directly on the moduli of objects the triviality of which one wants to prove.  
One important point here is that the triviality of a stratified bundle is 
recognized by the dimension of its space of stratified sections. This enables 
one to argue by semi-continuity arguments (see Lemma~\ref{semicontinuity}). 
Another point is that in the non-irreducible case, one has to deal with 
extensions which are not recognized on the moduli space. This also requires a 
new argument. See Section~\ref{sec:gen}.

\medskip

{\it Acknowledgements:}   The theorem for irreducible  stratified bundles was proved by Xiaotao Sun
on 2013  in a unpublished note. He told us on January of 2017 that he and his student  knew
how to prove  the general case based on an application of Simpson's theorem on moduli of framed bundles, which so far exists in
characteristic $0$ only.  We hope very much that he and his student
shall soon write  a characteristic $p > 0$ version of Simpson's theorem and that this way one shall have a 
proof  different from the one presented here. We thank him warmly for communicating his idea to us.

\section{Some properties of  vector bundles and their moduli}
We use the standard notations: if $f: Y\to X$ and $X'\to X$ are any morphisms, 
one denotes by $Y_{X'}\to X'$ the base change $Y\times_X X'\to X$.   If $Y$ has 
characteristic $p>0$, one denotes by $F_X:  X\to X$ the absolute Frobenius, by  
$f': Y'=Y\times_{X, F_X} X\to X$ the pull-back of $f$ by $F_X$, by $\xymatrix{  
\ar[dr]_f Y\ar[r]^{F_{Y/X}} & Y' \ar[d]^{f'} \\
 & X }$ the relative Frobenius above $X$.

 If $E$ is a vector bundle over $Y$, one denotes by $E_{X'}$ the pull-back 
vector bundle over $Y_{X'}$ via $Y_{X'}\to Y$. 

Recall that, from the theory of the Harder-Narasimhan filtration, any vector 
bundle $E$ on a smooth projective curve over an algebraically closed field has 
a unique maximal subbundle $F\subset E$ of maximal slope, and the slope of this 
subbundle $F$ is denoted by $\mu_{\rm max}(E)$. In particular, if $E$ has degree $0$, 
then $\mu_{\rm max}{\sE}\geq 0$, with equality precisely when $E$ is $\mu$-semistable.

\begin{lemma}\label{semicontinuity}   Let $C\to T$ be a smooth projective 
morphism, where $T$ is a Noetherian scheme and where the fibers  are  
geometrically connected smooth projective curves. Let $E$ 
be a vector bundle of rank $r$ on $C$. Then the following holds. 
\begin{enumerate}
\item[(i)] There is a constructible subset $T_0\subset T$ so that for a 
 point $t$ of 
$T$, it holds that
\[\mbox{$ H^0(C_t, E_t)\otimes_{k(t)} \sO_{C_t} \xrightarrow{\cong} E_t$ is a 
trivial bundle $\Leftrightarrow$ $t$ is a point of 
$T_0$.}\]
\item[(ii)] If moreover for every geometric point $t\in T$,  $E_t$ is a 
semistable vector bundle of degree 0, then $T_0$ is 
closed.  In particular, if $T_0$ contains a dense set of closed points of $T$, 
then $T_0=T$.
\item[(iii)] As $t$ ranges over the   points of $T$, the set of numbers 
$\mu_{\rm max}(E_t)$ is bounded.

\end{enumerate}
\end{lemma}
\begin{proof} The semicontinuity theorem \cite[Thm.~7.7.5]{EGAIII} gives (i), 
and also gives that for any geometric point 
$t$ in the closure of $T_0$, ${\rm dim}_{k(t)} H^0(C_t, E_t)\ge r$. 
For (ii), we note that ${\rm dim}_{k(t)} H^0(C_t, V)\le r$  for any semistable 
vector bundle $V$ of rank $r$ and degree $0$ on $C_t$, and equality holds if and 
only if $ H^0(C_t, V)\otimes_{k(t)} \sO_{C_t} \xrightarrow{\cong} V$.

For (iii), note that if $T$ is integral, $\eta$ is a  generic    point, 
then for an open dense subset $U\subset T$, and any point $t\in U$, one has  
 $\mu_{\rm max}(E_{\eta})=\mu_{\rm max}(E_t)$. Now (iii) follows in general 
by Noetherian induction.

\end{proof}

\begin{defn}\label{unipotent}
Let $X$ be a smooth projective variety over an algebraically closed field 
$k$, and let $E$ be 
a vector bundle of rank $r$ on $X$.
\begin{enumerate}
\item[(i)]  We say $E$ is {\em unipotent} if $E$ has a filtration by subbundles, 
the 
associated graded bundle of which  is a trivial vector bundle. 
\item[(ii)] We say $E$ is  $F_X$-{\em nilpotent of index} $N$  for some natural 
number  $ N>0$, if the $N$-th 
iterated Frobenius pullback $F_X^{N*}E$ is a trivial vector bundle.  
\end{enumerate}
\end{defn}

\begin{lemma}\label{Fnilpotent}
 Let $X$ be a smooth projective variety over an algebraically closed field $k$ 
of 
characteristic $p$, and let $E$ be a unipotent vector bundle of rank $r$, which 
is also 
$F_X$-nilpotent of some index. Then there is an $N=N(X,r)$, depending only the 
variety $X$ and the 
rank $r$, such that $E$ is $F_X$-nilpotent of index $N$. 
\end{lemma}
\begin{proof} We do induction on $r$, where the case   $r=1$ is trivial.
If $E$ is unipotent of rank $r>1$, there is an exact sequence
 \[0\to \sO_X\to E\to E'\to 0\]
 where $E'$ is unipotent of rank $r-1$. 
 Since $E$ is $F$-nilpotent of some index, it is easy to see that $E'$ is 
$F_X$-nilpotent of the 
same index, by taking that particular iterated Frobenius pullback of the above 
exact sequence, and 
noting that the quotient of a trivial bundle on  $X$  by a trivial subbundle is 
also a trivial 
bundle.
Hence by induction, there exists $N'=N'(X,r-1)$ so that $F_X^{N'*}E'$ is 
trivial. 
Then $F_X^{N'*}E$ determines an element of 
\[{\rm Ext}^1_X( H^0(X, F^{N'*}_XE')\otimes_k 
\sO_X, \sO_X)=H^0(X, F^{N' *}_X E')^\vee \otimes_kH^1(X,\sO_X).\]
The Frobenius $F_X$ sends a vector $v\otimes_k \gamma$ in this tensor product to 
$F^*_Xv\otimes_{F^*_k k} F_X^*\gamma$.
It therefore  suffices to note that if an element $\alpha\in H^1(X,\sO_X)$ is 
$F_X$-nilpotent, then 
$F_X^{g*}\alpha=0$, where $g=\dim_k H^1(X,\sO_X)$.
\end{proof}

We also note for reference the following two properties. The first one is an 
easy consequence of the fact that the slope of a vector bundle gets multiplied 
by $p$ under the Frobenius pull-back  and the second one is ultimately a consequence of this as well. 
\begin{lemma}\label{mu-max}
 Let $X$ be a smooth projective variety over an algebraically closed field $k$ 
of characteristic $p>0$, with a chosen polarization, and let $E$ be  a 
vector bundle with numerically trivial Chern classes. Let $E^{(m)}$ be a vector 
bundle  such that $F_X^{m*} E^{(m)}\cong 
E$, where $p^m>\mu_{\rm max}(E)\cdot {\rm rank}(E)$. Then $E^{(m)}$ is 
$\mu$-semistable. 
\end{lemma}

Recall that  on a smooth projective  scheme $X_S$ over a scheme $S$, a stratified bundle 
$\M$ is defined by an infinite Frobenius descent sequence  $(E^n, \tau^n)$,  
where $E^n$ is a vector bundle on the $n$-th Frobenius twist $X^{(n)}_S$ of 
$X_S$ over $S$ with isomorphisms $F^*_{X^{(n)}_S/S} E^{n} 
\xrightarrow{\tau^{n-1}} E^{n-1}$. 
If $S={\rm Spec}(k)$ where $k$ is a perfect characteristic $p>0$ field, as an 
object we can also write $\M=(E_m,\sigma_m)_{m\ge 0}$ where $E_m$ is a vector 
bundle on $X_k=X$ and   
$F^*_{X} E_n \xrightarrow{\sigma_{n-1}} E_{n-1}$. 
If $\M$ is given, and $n$ is a natural number, we denote by $\M(n)$ the 
stratified bundle  $\M(n)=(E_{n+m}, \sigma_{n+m})_{m \ge 0}$.

\begin{lemma}[\cite{EM10}, Proposition~3.2] \label{stable}
Let $X$ be a smooth projective variety. Any stratified bundle  is filtered by 
stratified subbundles $\M^i\subset \M^{i-1} \ldots  \subset \M^0=\M$ such that 
the associated graded stratified bundle  $ \oplus_i \M^i/\M^{i+1}$ has the 
property that  $\M^i/\M^{+1}$ is irreducible and for some $n$, all the 
underlying vector bundles of $(\M^i/\M^{i+1}) (n)$ are $\mu$-stable with trivial 
numerical Chern classes. 
\end{lemma}

In view of Lemma~\ref{mu-max} and Lemma~\ref{stable},  we shall consider the 
moduli scheme of $\mu$-stable  vector bundles. Given  a natural number $r>0$, 
$X_S\to S$ a smooth projective morphism, there is a coarse moduli 
quasi-projective scheme $M(r, X_S)\to S$ of $\mu$-stable  bundles   of rank $r$ 
with trivial numerical Chern classes, which universally corepresents the functor 
of families of geometrically stable bundles (\cite[Thm.~4.1]{Lan04}). In 
particular, if $T\to S$ is any morphism of schemes, the base change is an 
isomorphism
\ga{bc}{M (r, X_T)\xrightarrow{\cong} M(r, X_S)\times_S T.}
While applied to $F_S: T=S\to S$, this yields the isomorphism
\ga{fbc}{M(r, X'_S) \xrightarrow{\cong} M(r, X_S)\times_{S, F_S} S.}

We finally note the following, which is immediate from Lemma~\ref{Fnilpotent}.
\begin{cor}\label{Fnilpotent-cor}
If $X$ is a smooth projective variety over an algebraically closed field $k$ of charcateristic $p>0$, and  
$\M$ is a  stratified bundle, such that for the underlying sequence $(E_m,\sigma_m)$, each $E_m$ 
is an $F$-nilpotent vector bundle, then $\M$ is trivial as a stratified vector bundle.\end{cor}

\section{Reduction to the curve case}

\begin{lemma}\label{reductions}
It suffices to prove Theorem~\ref{Mainthm} in the special case when $Y=C$ is a 
smooth 
projective curve, and $X$ is a smooth projective surface.   
\end{lemma}
\begin{proof}
Let $C\subset Y$ be a nonsingular complete intersection of very ample divisors 
in $Y$. By \cite[Thm.~3.5]{ES16}, the homomorphism  $\pi_1^{\rm strat}(C)\to 
\pi_1^{\rm strat}(Y)$ is surjective.  
On the other hand, the 
homomorphism $\pi_1^{{\rm \acute{e}t}}(C)\to \pi_1^{\rm \acute{e}t}(X)$  factors 
through $\pi_1^{\rm \acute{e}t}(Y)\to 
\pi_1^{\rm \acute{e}t}(X)$, and so $C\to X$ satisfies the hypotheses of 
Theorem~\ref{Mainthm}.  This proves the first part of the lemma.  We assume now 
$Y=C$. 

On the other hand, we can make an embedding resolution of  the closed embedding 
$\pi(C)\subset X$ by a sequence of blow ups at closed points. 
So  the morphism $\pi:C\to X$ factors as a composition
\[C\stackrel{\tilde{\pi}}{\longrightarrow}\tilde{X}\stackrel{f}{\longrightarrow} 
X,\]
where $\tilde{\pi}:C\to \tilde{X}$ has smooth image, and $f:\tilde{X}\to X$ is a 
composition 
of blow ups at closed points.  Thus $f_*: \pi_1^{\rm \acute{e}t}(\tilde X)\to 
\pi_1^{\rm \acute{e}t}(X)$ is 
an isomorphism. 
Hence, it suffices to prove Theorem~\ref{Mainthm} for morphisms $\pi:C\to X$, 
where $C$ is a 
smooth curve, and $\pi(C)=D\subset X$ is a smooth curve as well. By  Bertini 
theorem  \cite[Thm.~6.3]{Joa83}, there is a smooth projective surface $X'\subset 
X$ which is a complete intersection 
of very ample divisors, and with $D\subset X'$. Then $\pi_1^{\rm 
\acute{e}t}(X')\to \pi_1^{\rm \acute{e}t}(X)$ is an 
isomorphism \cite[X, Thm.~3.10]{SGA2}. This finishes the proof.
\end{proof}

\section{Proof of Theorem~\ref{Mainthm} in the irreducible case} \label{sec:irr}
The aim of this section is to prove Theorem~\ref{Mainthm} for stratified bundles 
which are irreducible.  So $\pi: C\to X$ is a projective morphism between a 
smooth projective curve $C$ and a smooth projective surface over an 
algebraically closed field $k$ of characteristic $p>0$. 
As the triviality of  $\pi^*\M$ is equivalent to the triviality of $\pi^*\M(n)$ 
for any natural number $n$, 
Lemma~\ref{stable} enables us to assume that $\M =(E_m,\sigma_m)_{m\ge 0}$ has 
underlying  $E_m$ being $\mu$-stable with trivial
numerical Chern classes. 
 Let $\sS_X(r) $  be the set of isomorphism classes  $[\M]$ of 
irreducible stratified vector bundles $\M$ of rank $r$, all of whose underlying 
vector bundles are  $\mu$-stable. Let $S_r(n)\subset M(r,X)(k)$ be the set of 
all moduli points  $[E]\in M(r,X)(k)$ 
such that $E$ is one of the underlying  bundles of $\M(n)$, for some $[M]$ of 
$\sS_X(r)$.  We denote by 
$\overline{S_r(n)} \subset M(r,X)$ the Zariski closure,  which is a reduced 
closed subscheme. 
This defines a  decreasing sequence of closed subschemes
\[M_{r,X}\supset \overline{S_r}\supset 
\overline{S_r(1)}\supset\overline{S_r(2)}\supset\cdots\]
which, by Noetherianity,  becomes stationary at some finite stage. Set
\[\sN_r= \sN=\cap_{n\geq 1}\overline{S_r(n)}=\overline{S_r(N)}\mbox{ for some 
$N\geq 0$}.\]
The assignement   
\ga{Phi}{\Phi:\sN\to \sN, \   E\mapsto F_X^{*}E }
induces a dominant rational morphism  which does not commute with the structure 
morphism $\sN\to {\rm Spec} (k)$. 
One can  reinterpret  $\Phi$ as follows.  Let  $\sN'\subset M(r,X')$ 
be defined in the same way as $\sN\subset M(r, X)$.  Via the diagram 
\eqref{fbc}, one obtains a morphism
\ga{Nbc}{  F_k: \sN\times_{k, F_k} k \to \sN',}
which is a isomorphism over $k$ as we assumed $k$ to be algebraically closed.  
One has a factorization
\ga{PhiPhi'}{\xymatrix{\ar@/^2pc/[rr]^{\Phi} \ar[d] \sN  \ar[r]^{F_k^{-1}} & 
\ar[d]  \sN' \ar[r]^{\Phi'} & \sN \ar[dl]\\
{\rm Spec}(k) \ar[r]^{F_k^{-1}} & {\rm Spec}(k) 
}}
Now  $\Phi'$ is a dominant morphism of $k$-schemes
\ga{Phi'}{   \Phi' :\sN' \to \sN, \   [E]\mapsto [F_{X/k}^*E] .}

\medskip
Let $W$ be a scheme of finite type over $\F_p$ over which our data 
$(\pi: C\to X,  \ \sN'\subset M(r, X'),  \  \sN\subset M(r, X), \Phi')$ have a 
flat model.  We use the notation 
$(\pi_W: C_W\to X_W, \  \sN'_W\subset M(r, X_W), \  \sN_W\subset M(r, X_W), 
\Phi'_W: \sN'_W\to \sN_W)$, where we use \eqref{fbc} for $M(r,X_W)$ and $M(r, 
X'_W)$.   We assume in addition that $C_W\to W$ and $X_W\to W$ are smooth and projective.  The 
rational map of $W$-schemes 
\ga{Phi'W}{ \Phi'_W: \sN'_W\to \sN_W,  \ [E]\mapsto [F_{X_W/W}^*E] } 
is dominant.  So shrinking $W$ is necessary, for all closed points $t \in W$, 
the restriction $\Phi_{t}: \sN'_{t}\to \sN_{t}$ is a dominant rational map of 
$k(t)$-schemes which is defined by $[E]\mapsto [F^*_{X_{t}/t} E]$.  Let $p^a$ be 
the cardinality of $k(t)$. Then 
\ga{Phia}{\Phi^a_t: \sN_t\to \sN_t, \  [E]\mapsto [F^{a*}_{X_t/t}E]}
is rational dominant.  Using now  Hrushovski's essential theorem 
\cite[Cor.~1.2]{Hru04}, one proves  as in \cite[Thm.~3.14]{EM10}, by first  
taking a finite extension of  $k(t)$ so as to separate the components of 
$\sN_t$, and by  taking the corresponding power of $\Phi_t$ which stabilizes the 
components, that the set of closed moduli points $[E]\in \sN_t$ which are stable 
under a power of the Frobenius of the residue field  $k(t)$, and in particular 
are  trivialized by a finite \'etale cover of $X_t$, is dense.

\medskip

Let $T_W\to \sN_W$ be a surjective morphism  induced by a bundle $E$ on 
$X_W\times_W T_W$ such that the induced moduli map $T_W\to M(r, X_W)$ factors 
through $\sN_W\hookrightarrow M(r, X_W)$.  
Let $V=(\pi_W\times_W {\rm id}_T)^*E$ be the pull-back of $E$ on $C_W\times_W 
T_W$.  Lemma~\ref{semicontinuity} (i) applied  to $V$ and the projection 
$C_W\times_W T_W \to T_W$ implies that there is an open $T^0_W\subset T_W$ such 
that for all points $t \in T^0_W$, $V_t$ is the trivial bundle on $C_t$.  
Lemma~\ref{mu-max} together with Lemma~\ref{semicontinuity} (ii) imply that 
$T^0_W$ contains the pull-back in $T_W$ of a non-empty open dense subscheme 
$\sN^0\hookrightarrow \sN$. By Lemma~\ref{semicontinuity} (iii) applied to $V$ and 
the morphism $C_W\times_W T_W \to T_W$, and Lemma~\ref{mu-max}, we may assume $T^0_W$ 
contains $S_r(N_r)$ for some sufficiently large $N_r$. Thus the restriction to $C$ of any irreducible 
stratified bundle on $X$ of rank $r$ is trivial. This finishes the proof.

\begin{rmk}
We single out the last property for later use: there is a non-empty open dense 
subscheme $\sN^0\hookrightarrow \sN$ such that $S_r(N_r)\subset \sN^0$, and
for any  $ [E]\in \sN^0=:\sN^0_r$ (thus on $k$), $\pi^*E=H^0(C, \pi^*E)\otimes_k 
\sO_C$.

\end{rmk}



\section{Proof of Theorem~\ref{Mainthm} in the general case} \label{sec:gen}

The aim of this section is to prove Theorem~\ref{Mainthm} in the general case.  
The notation are as in Section~\ref{sec:irr}. Applying Lemma~\ref{stable}, we 
are reduced to considering 
all stratified bundles $\M=(E_m,\sigma_m)_{m\ge 0}$ on $X$ 
of some rank $r$,
which are  filtered by stratified subbundles $\M^i\subset \M^{i-1} \ldots  
\subset \M^0=\M$ such that the associated graded stratified bundle  
$\oplus_{i=0}^{s-1} \M^i/\M^{i+1}$ has the property that  $\M^i/\M^{+1} =(E_{im}, 
\sigma_{im})_{m\ge 0}$ is irreducible of rank say $r_i$, and  all the underlying vector bundles  
$E_{im}$ are $\mu$-stable with trivial numerical Chern classes.  By 
Theorem~\ref{Mainthm}  in the irreducible case, we then have the property that 
$\pi^* \M^i/\M^{i+1}$ is trivial.

\medskip
We now start the proof.  We fix a partition $r=r_0+\ldots +r_{s-1}$. We set  
\ga{Phi'product}{ \Phi':=\Phi'_0\times_k \ldots \times_k \Phi'_{s-1}: 
\sN'=\sN'_{r_0}\times_k \ldots \times_k \sN'_{r_{s-1}} \to 
\sN:=\sN_{r_0}\times_k \ldots \times_k \sN_{r_{s-1}}}
and denote by $U'=V'_0\times_k \ldots \times_k V'_{s-1} \subset \sN'$ the dense 
product open on which the dominant rational map $\Phi'$  is defined, and is given by 
$([E_0],\ldots,[E_{s-1}])\mapsto ([F_k^*E_0],\ldots,[F_k^*E_{s-1}])$, 
$V_i= F_k(V'_i)$ via \eqref{PhiPhi'} and $U=V_0\times_k \ldots \times_k V_{s-1}$.
We note that $\sN_{r_i}^0\subset V_i$ for each $0\leq i<s$.
\begin{defn} \label{defn:A}
In $\sN(k)$, define the set  
$$A=\{ ([E_0], \ldots, [E_{s-1}]) \in  \sN'(k)\} $$
 such that
\begin{itemize}
\item[1)]  there is a bundle $E$ with a filtration such that its associated 
graded is isomorphic to $\oplus_{i=1}^{s-1} E_i$;
\item[2)] $\pi^*E$ is not $F$-nilpotent.
\end{itemize}
We define $A' =F_k^{-1} (A)\subset \sN(k)$  via \eqref{PhiPhi'}. 
\end{defn}

\begin{lemma}\label{constructible}
The following properties hold.
\begin{enumerate}
\item[(i)] The set $A$ is a constructible subset of $\sN$.
\item[(ii)]  $\Phi'(U'\cap A')\subset A$,  \ $\Phi(U\cap A)\subset A$.
\end{enumerate}
\end{lemma} \label{lem:A}
\begin{proof} The property (ii) is by definition.
For (i), note that for $s=1$, there is nothing to prove $(A=\emptyset$). So we 
may assume $s\geq 
2$. 
Repeatedly using standard properties of Ext groups, we can find a surjective 
morphism $h:T\to \sN$ 
(where $T$ may not be connected), together with a vector bundle $E_T$ on 
$X\times_k T$, with the 
property that $E_T$ has an $s$-step filtration by subbundles, with quotients 
$E_{0,T}, \ldots,E_{s-1,T}$,  satisfying the following property: 
\begin{enumerate}
 \item[1)]  for each geometric point $t$ of $T$, the vector bundles 
$E_{0,t},\ldots,E_{s-1,t}$ 
obtained by restriction to $X_t$ are stable, of ranks $(r_0,\ldots,r_{s-1})$ 
respectively, and the 
induced morphism $T\to M(r_0,X)\times_k\cdots\times_k M(r_s,X)$
 determined by the sequence $(E_{0,T},\ldots, E_{s-1,T})$ 
regarded as a sequence of flat families of vector bundles on $X$  
is the given morphism $h$ composed with the inclusion  
$\sN\subset M(r_0,X)\times_k\cdots\times_k M(r_s,X)$;
\item[2)] if $E$ is any vector bundle on $X$, with an $s$-step filtration, and 
quotients 
$E_0,\ldots,E_{s-1}$ with $[E_i]\in \sN_{r_i}\subset M(r_i,X)$, then there is a 
point $t$ 
of $T$ such that $E\cong E_t$.
\end{enumerate}
   By Lemma~\ref{Fnilpotent}, combined with Lemma~\ref{semicontinuity},
the pull-back vector bundle $(\pi\times_k {\rm Id}_T)^*E_T$ on $C\times_k T$
has the property that  the set of 
points  $t\in T$, such that $\oplus_i\pi^*E_{i,t}$ is trivial and $\pi^*E_t$ is 
$F$-nilpotent, is 
a constuctible subset of $T$. This implies that its complement $\Sigma$ in $T$ 
is constructible as well. 
By definition,  
$A=h(\Sigma) \subset\sN$. Thus $A$ is constructible.  
\end{proof}
In order to finish the proof, it will be shown that $A$ is not dense in 
$\sN$ (a more precise assertion is given in Lemma~\ref{trick}. To this aim, we 
use Lemma~\ref{constructible} (ii)  to yet again define in it a  locus stable by $\Phi$. 
By construction, $U$ is a dense open subset of $\sN$, and $\Phi:U\to\sN$ is a 
morphism  with dense image. 
Set $U_1=U$, and for $m\geq 2$, inductively define dense open subsets 
$U_m\subset \sN$ by
\[U_m=\Phi^{-1}(U_{m-1}).\]
Then 
\[V_0\times_k \ldots \times_k V_{s-1}\supset U=U_1\supset U_2\supset U_3\supset 
\cdots\]
is a decreasing sequence of open sets in $\sN$. 
\begin{lemma}\label{trick} There is a natural number $m_1\neq 0$ such that   
$U_m\cap A=\emptyset$, for all $m\geq m_1$.  
\end{lemma}
\begin{proof}
Suppose $U_m\cap A\neq \emptyset$ for all $m$. Define
\[A_m=\overline{U_m\cap A}\subset A,\]
where the Zariski closure is taken relative to $A$. Then
\[A=A_0\supset A_1\supset A_2\supset\cdots\]
is a decreasing sequence of closed, nonempty subsets. Hence by the Noetherian 
property, there 
exists $m_0$ so that 
\[A_{m_0}=A_{m_0+1}=\cdots=:A_{\infty}\]
is a nonempty closed subset of $A$. 
By construction, $U_r\cap A_m$ is dense for any $m\geq r$, and so the rational 
map $\Phi^r$ is 
defined as a rational map on $A_m$, for all $m\geq r$, and satisfies 
$\Phi^r(A_m)\subset A_{m-r}$. 
Hence $\Phi^r:A_{\infty}\to A_{\infty}$ is a well-defined rational map, for each 
$r\geq 0$. 
Define 
\[B_r=\overline{\Phi^r(A_{\infty})}\neq \emptyset,\]
where the closure is again relative to $A$. Then
\[A_{\infty}=B_0\supset B_1\supset B_2\supset \cdots\]
is a decreasing sequence of nonempty closed subsets. Hence there exists $m_1$ 
such that 
\[B_{m_1}=B_{m_1+1}=\cdots:=B_{\infty}\neq \emptyset\]
is a closed subset of $A$. 
By construction,  $\Phi$ is a rational map defined on $B_r$, for each $r\geq 0$, 
and further satisfies
\[\overline{\Phi(B_r)}=\overline{\Phi(\overline{\Phi^r(A_{\infty})}}=
\overline{\Phi(\Phi^r(A_{\infty})}=\overline{\Phi^{r+1}(A_{\infty})}=B_{r+1}.\]
The first equality is by definition, the second equality is true for all 
continuous maps of topological spaces, the third one is again by definition. 
Hence $\Phi:B_{\infty}\to B_{\infty}$ is a  dominant  rational self-map.

\medskip 

We now define $B_{\infty}' =F_k(B_\infty)  \subset \sN'$  via \eqref{PhiPhi'}. 
Then 
\ga{Phi'B}{  \Phi': B'_\infty\to B_\infty, \ \oplus_{i=0}^{s-1} [E_i] \to  
\oplus_{i=0}^{s-1}[F^*_{X/k}E_i]}
is a dominant rational map. 
One defines $T'$ as the Frobenius twist of $T$,  and the bundle $E'$ on 
$T\times_k X'$ as the pull back of $E$ on $X\times_kT$ via the Frobenius twist 
$F_k: X'\times_k T'\to X\times_k T$. 
Then  \eqref{Phi'B} implies that the bundle
$(F^*_{X'\times T'_{B_\infty'}/T'_{B'_\infty}})^* E_{X'\times_k T_{B'_\infty}}$ 
is defined on $X\times_k T_{B_\infty}$ outside of a locus $X\times_k \Sigma$ 
with $\Sigma\subset T_{B_\infty}$ mapping to a codimension $\ge 1$ constructible 
subset of $B_\infty$. 

\medskip

Applying again Hrushovski's theorem \cite[Cor.~1.2]{Hru04} as in 
Section~\ref{sec:irr} to \eqref{Phi'B} replacing 
\eqref{Phi'}, we obtain a model $B_{\infty, W}$ of $B_\infty$,  and for any 
closed point of $t \in W$   a dense set of closed points
$e=\oplus_{i=0}^{s-1} [E_i]$  in $B_{\infty, t}$ which are stabilized by a power 
 of the Frobenius of $k(t)$. 
Choosing now $W$ such that $h: T_{B_\infty}\to B_{\infty}$ has a model 
$h_W:  T_{B_\infty,W}\to B_{\infty,W}$, the fiber $T_e$  above $e \in B_{\infty, 
t}$ where $t$ is a closed point of $W$  has the property:
\begin{itemize}
\item[] for any closed point $\tau \in T_e$,  there is at least another closed 
point $\tau' \in T_e$ such that 
 $F^{a(e)*}_{X_t} E_{X_t\times_t \tau} $, 
 where $d(t)$  divides the degree of $t$,  is isomorphic to 
  $F^{d(t)a(e)*}_{X_t} E_{X_t\times_t \tau'} $ .

\end{itemize}
Indeed, we know that  for any closed point $\tau$, the moduli point of the 
graded bundle of 
$F^{a(e)*}_{X_t} E_{X_t\times_t \tau} $ in $B_{\infty, \tau}$ is isomorphic over 
$\bar \F_p$ to the graded bundle associated to 
$E_{X_t\times_t \tau} $, and thus there are isomorphic over $\tau $ (see 
\cite[Lem.~2.4]{EK16}).
By finiteness of the rational points of $T_e$ over finite extensions of $k(e)$, 
we conclude that   for any closed point $e \in B_{\infty, t}$,  there is a 
closed point $\tau \in T_e$  such that 
$F^{b}_{X_t} E_{X_t\times_t \tau} $ for some natural number $b\ge 1$. This shows 
in particular  that
$E_{X_t\times_t \tau} $ is trivialized by a finite \'etale cover. 
We now argue as in Section~\ref{sec:irr}. 

\medskip

Let $V=(\pi_W\times_W {\rm id}_{T_{B_\infty}})^*E$ be the pull-back of $E$ on 
$C_W\times_W T_{B_\infty, W}$.  Lemma~\ref{semicontinuity} (i) applied  to $V$ 
and the projection $C_W\times_W T_{B_\infty, W} \to T_{B_\infty, W}$ implies 
that there is a constructible subset  $T^0 \subset T_{B_\infty,W}$,  mapping 
surjectively to all closed points of  $B_{\infty, W}$,   such that for all 
points $t \in T^0 $, $V_t$ is the trivial bundle on $C_t$.
Thus this constructible set contains points of $B_\infty$, which contradicts the 
definition of $B_\infty$. 

\medskip
 This finishes the proof. 
\end{proof}

Now if $\M$ is any stratified bundle on $X$ of rank $r$ with irreducible filtered
constituents $\M^i/\M^{i+1}$ of ranks $r_i$, for $0\leq i<s$, we may, after replacing 
$\M$ by some $\M(n)$ if necessary, assume that the associated sequences of vector bundles 
$E_{im}$ (for $m\geq 0$) satisfy that $(\ldots,[E_{im}],\ldots)\in \sN^0\cap U_{m_1}$, where 
$m_1$ is as in Lemma~\ref{trick}, and hence  lies in the complement of $A$. Thus $\M$ has 
corresponding sequence $(E_m,\sigma_m)$ such that $\pi^*E_m$ is $F$-nilpotent on $C$. 
Hence by Corollary~\ref{Fnilpotent-cor},  $\pi^*\M$ is trivial as a stratified bundle.

\end{document}